\documentclass[11pt]{article}

\usepackage[margin=1.15in]{geometry}
\usepackage[T1]{fontenc}
\usepackage{lmodern}

\usepackage{amsmath, amssymb, amsthm, mathtools}

\usepackage[round,authoryear]{natbib}
\usepackage[colorlinks=true,linkcolor=blue,citecolor=blue,urlcolor=blue]{hyperref}

\numberwithin{equation}{section}

\newtheorem{proposition}{Proposition}[section]

\newtheorem{condition}{Condition}[section]
\newenvironment{customcon}[1]
{\renewcommand\thecondition{#1}\condition}
{\endcondition}

\newcommand{\M}{\mathcal{M}}
\newcommand{\W}{\mathcal{W}}
\newcommand{\R}{\mathbb{R}}

\newcommand{\E}{\mathrm{E}}

\newcommand{\dd}{\,d}
\DeclareMathOperator*{\argmin}{\arg\min}

\begin{document}

\begin{center}
{\Large\bfseries A note on a local entropy condition in the Wasserstein space \par}

\vspace{1em}

Chang Jun Im\textsuperscript{1},
Jeong Min Jeon\textsuperscript{2,*},
Byeong U. Park\textsuperscript{3}

\vspace{0.75em}

{\small
\textsuperscript{1}The Institute for Data Innovation in Science, Seoul National University\\
\textsuperscript{2}Department of Statistics and School of Transdisciplinary Innovations, Seoul National University\\
\textsuperscript{3}Department of Statistics, Seoul National University
}

\vspace{0.75em}

{\small *Corresponding author [\texttt{jm.jeon@snu.ac.kr}]}
\end{center}

\vspace{1em}

\begin{abstract}
Entropy conditions are widely used in empirical-process analyses to establish asymptotic properties of M-estimators. In regression problems with responses taking values in a general metric space, a commonly imposed condition requires the entropy integral associated with a shrinking ball centered at the target object to remain uniformly bounded as the ball radius tends to zero. We show that this condition generally fails in the quadratic Wasserstein space of univariate probability distributions supported on a compact interval. Specifically, when the target object is a strictly increasing and absolutely continuous distribution function whose derivative is bounded away from zero and infinity, the corresponding entropy integral diverges as the ball radius tends to zero. The same divergence persists even when the ambient space is restricted to the class of distribution functions satisfying fixed uniform two-sided Lipschitz bounds. These results indicate that different asymptotic analyses are required for the quadratic Wasserstein space.
\end{abstract}



\section{Introduction}

Non-Euclidean data analysis has become an important part of modern statistics as data objects increasingly take values in spaces that lack a Euclidean vector space structure \citep{Chang1986,Hein2009,DiMarzioPanzeraTaylor2014,JeonParkVanKeilegom2021,stocker2023,JeonVanBever2025,JeonVanBever2026}. Examples include spherical vectors, probability distributions, covariance matrices, networks and shapes. In such settings, standard averages and least-squares regression are often unavailable or geometrically inappropriate. We refer to \citet{MarronAlonso2014} and \citet{dubey2024} for reviews of non-Euclidean data analysis.

Let \(\M\) be a totally bounded metric space with metric \(d_\M\). Suppose that \(Y\) is an \(\M\)-valued response and \(X\) is a predictor. A natural regression target in this setting is the conditional Fr\'echet mean of \(Y\) given \(X=x\), defined as
\[
  m_\oplus(x):=\argmin_{y\in\M}\E(d_\M^2(Y,y)\mid X=x).
\]
Recently, \citet{PetersenMuller2019} introduced global and local regression methods to estimate $m_\oplus(x)$. Subsequent methodological and theoretical developments include functional models for time-varying random objects \citep{DubeyMuller2020}, uniform convergence theory for local Fr\'echet regression \citep{ChenMuller2022}, variable selection for global Fr\'echet regression \citep{TuckerWuMuller2023}, single-index Fr\'echet regression \citep{BhattacharjeeMuller2023}, nonlinear global Fr\'echet regression \citep{Bhattacharjee2025}, partially global Fr\'echet regression \citep{TuckerWu2025} and deep learning for Fr\'echet regression \citep{iao2025}.

For \(m\in\M\) and \(r>0\), write
\[
B_\M(m,r):=\{y\in\M:d_\M(m,y)<r\}.
\]
For any set \(S\subset\M\), let \(N(r,S,d_\M)\) denote the minimal number of open \(d_\M\)-balls of radius \(r\), centered in \(S\), required to cover \(S\). Many studies including the aforementioned works employ the following entropy condition or its variants in their asymptotic analyses:

\begin{customcon}{A}\label{con:A}
At a fixed predictor value \(x\), the following holds:
\begin{equation*}
    \int_0^1 \sqrt{1+\log N(\delta \epsilon,B_\M(m_\oplus(x),\delta),d_\M)}\dd \epsilon=O(1)\quad\text{as}\quad\delta\to0.
\end{equation*}
\end{customcon}

Condition~\ref{con:A} or its uniform-in-$x$ variants are used to
control the stochastic fluctuations of the estimators of $m_\oplus$ through
empirical-process arguments. If this integral diverges, such arguments may yield convergence rates slower than the usual rates derived under Condition~\ref{con:A}.

For univariate distributions, the quantile representation provides an isometric embedding
of the quadratic Wasserstein space into \(L^2(0,1]\). Since
quantile functions are monotone, one might 
expect that standard entropy bounds for monotone function classes
provide Condition~\ref{con:A} when $\mathcal{M}$ is the quadratic Wasserstein space. Recently, \citet{PetersenMuller2026} proved that this approach establishes Condition~\ref{con:A} when $\mathcal{M}$ is restricted to a finite-dimensional subset of the quadratic Wasserstein space, while noting that the condition may fail on the full quadratic Wasserstein space. In this paper, we provide two counterexamples showing that Condition~\ref{con:A} can fail both on the full quadratic Wasserstein space and on some regular subset thereof. Specifically, we prove that, when $m_\oplus(x)$ is any strictly increasing and absolutely continuous distribution function whose derivative is bounded away
from zero and infinity, the integral in Condition~\ref{con:A} is bounded below
by a constant multiple of \(\sqrt{\log(1/\delta)}\) for all sufficiently small \(\delta\). Moreover, the same failure occurs even after restricting $\mathcal{M}$ to a class of strictly increasing and absolutely continuous distribution functions satisfying fixed uniform two-sided Lipschitz bounds. Since the quadratic Wasserstein space is an important response space in Fr\'echet regression, this failure can pose a significant issue.

The rest of this paper is organized as follows.
Section~\ref{sec:first} gives a counterexample for the compactly
supported Wasserstein space. Section~\ref{sec:second} shows that the same
phenomenon persists on a regular subset of the Wasserstein space. Section~\ref{sec:discussion} discusses possible ways to address this issue.

\section{First counterexample}\label{sec:first}

For \(-\infty<a<b<\infty\), let \(\W([a,b])\) denote the set of all distribution functions \(F:\R\to[0,1]\) whose quantile function
\[
  F^{-1}(t)=\inf\{u\in\R:F(u)\ge t\},\quad t\in(0,1],
\]
takes values in \([a,b]\). Endow \(\W([a,b])\) with the quadratic Wasserstein metric defined as
\begin{equation}\label{eq:wass-dist}
    d_{\W([a,b])}(F,G)=\left(\int_0^1\big(F^{-1}(t)-G^{-1}(t)\big)^2\dd t\right)^{1/2}.
\end{equation}
It is well known that the quadratic Wasserstein metric space \((\W([a,b]),d_{\W([a,b])})\) is totally bounded \citep{PanaretosZemel2020}. Without loss of generality, we focus on \(\W:=\W([0,1])\). Write
\begin{equation}\label{eq:H}
    H(F,\delta):=\int_0^1 \sqrt{1+\log N(\delta \epsilon,B_\W(F,\delta),d_\W)}\dd \epsilon.
\end{equation}
Note that the integral in Condition~\ref{con:A} equals $H(m_\oplus(x),\delta)$. The following proposition shows that Condition \ref{con:A} can fail when $(\M,d_\M)=(\W,d_\W)$.

\begin{proposition} \label{thm:first}
    Let $F \in \mathcal{W}$ be a distribution function that is strictly increasing and absolutely continuous on $[0,1]$. Suppose that there exist constants \(0<c_0\le c_1<\infty\) such that $c_0\le F'(u)\le c_1$ for almost everywhere $u\in(0,1)$. Then, $H(F,\delta) \to \infty$ as $\delta \to 0$.
\end{proposition}

\begin{proof}
    Fix a small $\delta > 0$. Consider a sequence $\{a_l: l \ge 0\}$ defined as
    \begin{align*}
        a_l \equiv a_l(\delta):= \left( \frac{3\delta^2}{4c_1} \right)^{1/3} \times l.
    \end{align*}
    Let $L \equiv L(\delta):= \max \{l \in \mathbb{N} : a_l \leq 1 \}$. For $l \in \{1, \ldots, L\}$, define $F_l \equiv F_l(\cdot,\delta)\in \mathcal{W}$ as
    \begin{align*}
        F_l (u) :=
        \begin{cases}
            0, & \text{if}~u < 0, \\
            F(u), & \text{if}~u\in[0,a_{l-1})\cup[a_l,\infty),\\
            F(a_{l-1}), & \text{if}~u \in [a_{l-1}, a_l). 
        \end{cases}
    \end{align*}
    Then, for all $t\in(0,1]$,
    \begin{align*}
        F_l^{-1} (t) =
        \begin{cases}
            a_{l}, & \text{if}~t \in (F(a_{l-1}), F(a_l)], \\
            F^{-1}(t), & \text{if}~t \notin (F(a_{l-1}), F(a_l)].
        \end{cases}
    \end{align*}
    Hence,
    \begin{align*}
        d_{\mathcal{W}} \left(F, F_l \right)^2 & = \int_0^1 (F^{-1}(t)- F_l^{-1}(t))^2 \mathrm{d}t
        = \int_{F(a_{l-1})}^{F(a_l)} \left(F^{-1}(t) - a_l \right)^2 \mathrm{d}t = \int_{a_{l-1}}^{a_l} \left(u - a_l \right)^2 F'(u) \mathrm{d}u \\
        &\le \frac{c_1}{3} (a_{l} - a_{l-1})^3
        = \frac{\delta^2}{4}.
    \end{align*}
    Fix $\epsilon \in (0, \sqrt{c_0/(8c_1)})$. For all $1 \le l<m\le L$, we get
    \begin{align}\label{dist-Fmn}\begin{split}
        d_{\mathcal{W}} \left( F_m,  F_l \right)^2
        & = \int_{F(a_{l-1})}^{F(a_l)} \left( F^{-1}(t) - a_l \right)^2 \mathrm{d}t + \int_{F(a_{m-1})}^{F(a_m)} \left( F^{-1}(t) - a_m \right)^2 \mathrm{d}t \\
        & \geq \frac{c_0}{3}\big((a_{l} - a_{l-1})^3 + (a_{m} - a_{m-1})^3\big)
        = \frac{c_0}{c_1} \cdot \frac{\delta^2}{2} \geq 4 \delta^2 \epsilon^2,
        \end{split}
    \end{align}
    where the last inequality follows from $c_0/c_1 \ge 8 \epsilon^2$.
    If $G \in B_{\mathcal{W}}(F_l, \delta \epsilon)$, then
    $$d_{\mathcal{W}} \left(G,F \right) \leq d_{\mathcal{W}} \left( G,F_l \right) + d_{\mathcal{W}} \left(F_l,F \right) < \delta \epsilon + \delta / 2 < \delta,$$
    so that
    \begin{align}\label{covering1}
        \bigcup_{l=1}^L B_{\mathcal{W}}(F_l, \delta \epsilon) \subset B_{\mathcal{W}}(F, \delta).
    \end{align}

    We claim that
    \begin{align}\label{claim-entropy}
    L=N\Big(\delta\epsilon,\,\bigcup_{l=1}^L B_{\mathcal{W}}(F_l, \delta \epsilon),\,d_{\mathcal{W}}\Big).
    \end{align}
    Then, \eqref{claim-entropy} together with \eqref{covering1} implies that
    \begin{align*}
    L\leq N^{\rm ext}\Big(\delta\epsilon/2,\,\bigcup_{l=1}^L B_{\mathcal{W}}(F_l, \delta \epsilon),\,d_{\mathcal{W}}\Big)
    \leq N^{\rm ext}(\delta\epsilon/2,B_{\mathcal{W}}(F, \delta),d_{\mathcal{W}})\leq N(\delta\epsilon/2,B_{\mathcal{W}}(F, \delta),d_{\mathcal{W}}),
    \end{align*}
    where $N^{\rm ext}$ denotes the external covering number.
    This implies that
    \begin{align*}
        \int_{0}^{1} \sqrt{1+\log N(\delta \epsilon, B_{\mathcal{W}}(F, \delta), d_{\mathcal{W}})} \, \mathrm{d} \epsilon 
        &= \frac{1}{2} \int_{0}^2 \sqrt{1+\log N(\delta \epsilon/2, B_{\mathcal{W}}(F, \delta), d_{\mathcal{W}})} \, \mathrm{d} \epsilon \\
        &\ge \frac{1}{4\sqrt{2}}\sqrt{\frac{c_0}{c_1}}\sqrt{1+\log\Big\lfloor\Big(\frac{4c_1}{3\delta^2}\Big)^{1/3}\Big\rfloor},
    \end{align*}
    where the last inequality follows from $\sqrt{c_0/(8c_1)}\le 2$, and $\left\lfloor \cdot \right\rfloor : \mathbb{R} \to \mathbb{Z}$ denotes the floor function.

    It remains to prove \eqref{claim-entropy}. Clearly,
    \[
    L \ge N\Big(\delta\epsilon,\,\bigcup_{l=1}^L B_{\mathcal{W}}(F_l, \delta \epsilon),\,d_{\mathcal{W}}\Big).
    \]
    To prove the reverse inequality, suppose that there exist $G_1, \ldots, G_r \in \bigcup_{l=1}^L B_{\mathcal{W}}(F_l, \delta \epsilon)$ with $r <L$ such that
    $\{B_{\mathcal{W}}(G_1,\delta \epsilon), \ldots, B_{\mathcal{W}}(G_r,\delta \epsilon)\}$ covers $\bigcup_{l=1}^L B_{\mathcal{W}}(F_l, \delta \epsilon)$. Then,
    $F_1, \ldots, F_L \in \bigcup_{j=1}^r B_{\mathcal{W}}(G_j, \delta \epsilon)$. Since $r <L$, there exist an index $j$ and distinct indices $l$ and $m$ such that $F_l, F_m \in B_{\mathcal{W}}(G_j,\delta \epsilon)$. This means that $d_{\mathcal{W}}(F_l,F_m) <2 \delta \epsilon$. This contradicts \eqref{dist-Fmn}, which proves
    \[
    L \le N\Big(\delta\epsilon,\,\bigcup_{l=1}^L B_{\mathcal{W}}(F_l, \delta \epsilon),\,d_{\mathcal{W}}\Big).
    \]
    This completes the proof of the proposition.
\end{proof}

The class of distribution functions satisfying the assumptions of
Proposition~\ref{thm:first} is broad. Hence, Condition~\ref{con:A} fails for
various values of \(m_\oplus\).

\section{Second counterexample}\label{sec:second}

One might ask whether Condition~\ref{con:A} can hold when the ambient space
is restricted to a more regular subset of \(\W\). For fixed constants \(0<L<U<\infty\), define
\begin{align*}
    \mathcal{W}_{L, U} := \{ F \in \W : L |u-v| \leq \left| F(u) - F(v) \right| \leq U |u-v| \text{~for all~}u, v \in [0, 1] \}.
\end{align*}
Note that every distribution function in $\mathcal{W}_{L, U}$ is strictly increasing and absolutely continuous on $[0, 1]$. The next proposition shows that the answer remains negative for the metric space \((\W_{L,U}, d_\W)\). Write
\[
H_{L, U}(F,\delta):=\int_{0}^{1} \sqrt{1+\log N(\delta \epsilon, B_{\mathcal{W}_{L, U}}(F, \delta), d_{\mathcal{W}})} \, \mathrm{d} \epsilon.
\]

\begin{proposition} \label{thm:second}
    Let \(c_0\) and \(c_1\) satisfy \(L<c_0<c_1<U\), and let
\(F\in\W_{L,U}\) be a distribution function such that $c_0|u-v|\le |F(u)-F(v)|\le c_1|u-v|$ for all $u,v\in[0,1]$. Then, $H_{L, U}(F,\delta) \to \infty$ as $\delta \to 0$.
\end{proposition}

\begin{proof}
    Fix a small $\delta > 0$. Consider sequences $\{a_l: l \ge 0\}$ and $\{b_l: l \ge 1\}$ defined as
    \begin{align*}
        a_l &\equiv a_l(\delta):= \left( \frac{27\delta^2}{16U} \right)^{1/3} \times l, \\
        b_l &\equiv b_l(\delta):= \frac{U a_{l} - L a_{l-1} + F(a_{l-1}) - F(a_{l})}{U-L}.
    \end{align*}
    Define $K \equiv K(\delta):= \max \{l \in \mathbb{N} : a_l \leq 1 \}$. For $l \in \{1, \ldots, K\}$, it follows that
    \begin{align*}
        L (a_{l} - a_{l-1}) < c_0 (a_{l} - a_{l-1}) \le F(a_{l}) - F(a_{l-1}) \le c_1 (a_{l} - a_{l-1}) < U (a_{l} - a_{l-1}),
    \end{align*}
    which implies that $a_{l-1} < b_l < a_l$. For $l \in \{1, \ldots, K\}$, define $F_l \equiv F_l(\cdot,\delta)\in\mathcal{W}_{L,U}$ as
    \begin{align*}
        F_l (u) :=
        \begin{cases}
            0, & \text{if}~u < 0, \\
            F(u), & \text{if}~u\in[0,a_{l-1})\cup[a_l, \infty),\\
            L(u-a_{l-1}) + F(a_{l-1}), & \text{if}~u \in [a_{l-1}, b_l), \\
            U(u-a_{l}) + F(a_{l}), & \text{if}~u \in [b_l, a_l). 
        \end{cases}
    \end{align*}
    The quantile function of $F_l$ is given by
    \begin{align*}
        F_l^{-1} (t) =
        \begin{cases}
            \frac{t-F(a_{l-1})}{L} + a_{l-1}, & \text{if}~t \in (F(a_{l-1}), F_l(b_l)], \\
            \frac{t-F(a_{l})}{U} + a_{l}, & \text{if}~t \in (F_l(b_{l}), F(a_l)], \\
            F^{-1}(t), & \text{if}~t \notin (F(a_{l-1}), F(a_l)].
        \end{cases}
    \end{align*}
    A direct calculation shows that
    \begin{align*}
        \frac{t-F(a_{l-1})}{U} + a_{l-1} < \frac{t-F(a_{l-1})}{c_1}+ a_{l-1} &\le F^{-1} (t) \le \frac{t-F(a_{l})}{c_1} + a_{l} \le \frac{t-F(a_{l})}{U} + a_{l}, \\
        \frac{t-F(a_{l-1})}{U} + a_{l-1} &\le F_l^{-1} (t) \le \frac{t-F(a_{l})}{U} + a_{l}
    \end{align*}
    for all $t\in(F(a_{l-1}),F(a_l)]$.
    This implies that
    \begin{align}\label{diff bound}
    \begin{split}
    |F^{-1}(t)-F^{-1}_l(t)| & < \left|\frac{t-F(a_{l})}{U} + a_{l}-\left(\frac{t-F(a_{l-1})}{U} + a_{l-1}\right) \right| \\
    & = \left|a_{l} - a_{l-1} - \frac{F(a_l) -F(a_{l-1})}{U}\right|.
    \end{split}
    \end{align}
    It follows that
    \begin{align*}
        d_{\mathcal{W}} \left(F, F_l \right)^2 & = \int_0^1 (F^{-1}(t)- F_l^{-1}(t))^2 \mathrm{d}t
        = \int_{F(a_{l-1})}^{F(a_l)} \left(F^{-1}(t) - F_l^{-1}(t) \right)^2 \mathrm{d}t \\
        &< \int_{F(a_{l-1})}^{F(a_l)} \left(a_{l} - a_{l-1} - \frac{F(a_l) -F(a_{l-1})}{U}\right)^2 \mathrm{d}t \\
        & = \left(a_{l} - a_{l-1} - \frac{F(a_l) -F(a_{l-1})}{U}\right)^2 \left( F(a_l) -  F(a_{l-1}) \right) \\
        & = \frac{1}{2U^2} \left(U (a_{l} - a_{l-1}) - (F(a_l) -F(a_{l-1})) \right)^2 \left( 2( F(a_l) -  F(a_{l-1}) ) \right) \\
        & \le \frac{1}{2U^2} \left( \frac{2U(a_{l} - a_{l-1})}{3} \right)^3 = \frac{4U}{27} (a_{l} - a_{l-1})^3 = \frac{\delta^2}{4},
    \end{align*}
    where the first inequality follows from \eqref{diff bound}, and the last follows from the arithmetic--geometric mean inequality applied to three nonnegative numbers. Now, define
    \begin{align*}
        \epsilon_0 := \min \left\{ \left(\frac{9L^3}{128U}\right)^{1/2} \cdot \min \left\{ \frac{1}{L} - \frac{1}{c_0} , \frac{1}{c_1}- \frac{1}{U} \right\}, \frac{1}{2} \right\}
    \end{align*}
    and let $\epsilon \in \left(0, \epsilon_0 \right)$. Since
    \begin{align*}
        F^{-1} (t) \le \frac{t-F(a_{l-1})}{c_0} + a_{l-1} < F_l^{-1} (t), \quad &\text{if}~t \in (F(a_{l-1}), F_l(b_l)], \\
        F^{-1} (t) \le \frac{t-F(a_{l})}{c_1} + a_{l} \le F_l^{-1} (t), \quad &\text{if}~t \in (F_l(b_l), F_l(a_l)],
    \end{align*}
    we have
    \begin{align}\label{some bounds}
    \begin{split}
    F_l^{-1}(t)-F^{-1}(t)\ge F_l^{-1}(t)-\left(\frac{t-F(a_{l-1})}{c_0} + a_{l-1}\right), \quad &\text{if}~t \in (F(a_{l-1}), F_l(b_l)], \\
    F_l^{-1}(t)-F^{-1}(t)\ge F_l^{-1}(t)-\left(\frac{t-F(a_{l})}{c_1} + a_{l}\right), \quad &\text{if}~t \in (F_l(b_l), F_l(a_l)].
    \end{split}
    \end{align}
    It follows that
    \begin{align}\begin{split} \label{eq:7.33_2}
        & \int_{F(a_{l-1})}^{F(a_l)} \left(F^{-1}(t) - F_l^{-1}(t) \right)^2 \mathrm{d}t \\
        & = \int_{F(a_{l-1})}^{F_l(b_l)} \left(F^{-1}(t) - F_l^{-1}(t) \right)^2 \mathrm{d}t + \int_{F_l(b_{l})}^{F(a_l)} \left(F^{-1}(t) - F_l^{-1}(t) \right)^2 \mathrm{d}t \\
        & \ge \int_{F(a_{l-1})}^{F_l(b_l)} \left( \frac{1}{L} - \frac{1}{c_0} \right)^2 \left( t-F(a_{l-1}) \right)^2 \mathrm{d}t + \int_{F_l(b_{l})}^{F(a_l)} \left( \frac{1}{c_1} - \frac{1}{U} \right)^2 \left( t-F(a_{l}) \right)^2 \mathrm{d}t \\
        & = \frac{1}{3} \left( \frac{1}{L} - \frac{1}{c_0} \right)^2 \left( F_l(b_l)-F(a_{l-1}) \right)^3 + \frac{1}{3} \left( \frac{1}{c_1} - \frac{1}{U} \right)^2 \left( F(a_l)-F_l(b_l) \right)^3 \\
        & \ge \frac{1}{3} \min \left\{ \left( \frac{1}{L} - \frac{1}{c_0} \right)^2, \left( \frac{1}{c_1} - \frac{1}{U} \right)^2 \right\} \left( \left( F_l(b_l)-F(a_{l-1}) \right)^3 + \left( F(a_l)-F_l(b_l) \right)^3 \right) \\
        & \ge \frac{1}{12} \min \left\{ \left( \frac{1}{L} - \frac{1}{c_0} \right)^2, \left( \frac{1}{c_1} - \frac{1}{U} \right)^2 \right\} \left( F(a_l)-F(a_{l-1}) \right)^3 \\
        & > \frac{L^3}{12} \min \left\{ \left( \frac{1}{L} - \frac{1}{c_0} \right)^2, \left( \frac{1}{c_1} - \frac{1}{U} \right)^2 \right\} \left( a_l-a_{l-1} \right)^3 \\
        & = \frac{9L^3}{64U} \min \left\{ \left( \frac{1}{L} - \frac{1}{c_0} \right)^2, \left( \frac{1}{c_1} - \frac{1}{U} \right)^2 \right\} \delta^2 \ge 2 \delta^2 \epsilon_0^2 > 2 \delta^2 \epsilon^2,
    \end{split} \end{align}
    where the first inequality follows from \eqref{some bounds} and the third inequality follows from the fact that $4(a^3+b^3)\ge(a+b)^3$ for all $a,b>0$.
    By \eqref{eq:7.33_2}, for all $1\leq l<m\leq K$, we have
    \begin{align}\begin{split} \label{eq:7.33_3}
        d_{\mathcal{W}} \left( F_m,  F_l \right)^2
        & = \int_{F(a_{l-1})}^{F(a_l)} \left( F^{-1}_m(t) - F_l^{-1}(t) \right)^2 \mathrm{d}t + \int_{F(a_{m-1})}^{F(a_m)} \left( F^{-1}_m(t) - F_l^{-1}(t) \right)^2 \mathrm{d}t \\
        & = \int_{F(a_{l-1})}^{F(a_l)} \left( F^{-1}(t) - F_l^{-1}(t) \right)^2 \mathrm{d}t + \int_{F(a_{m-1})}^{F(a_m)} \left( F^{-1}(t) - F_m^{-1}(t) \right)^2 \mathrm{d}t > 4 \delta^2 \epsilon^2.
    \end{split} \end{align}
    Arguing as in the proof of Proposition \ref{thm:first}, we get
    \begin{align*}
        K=N\Big(\delta\epsilon,\,\bigcup_{l=1}^K B_{\mathcal{W}_{L, U}}(F_l, \delta \epsilon),\,d_{\mathcal{W}}\Big) \le N(\delta\epsilon/2,B_{\mathcal{W}_{L, U}}(F, \delta),d_{\mathcal{W}}).
    \end{align*}
    This implies that
    \begin{align*}
        \int_{0}^{1} \sqrt{1+\log N(\delta \epsilon, B_{\mathcal{W}_{L, U}}(F, \delta), d_{\mathcal{W}})} \, \mathrm{d} \epsilon
        &= \frac{1}{2} \int_{0}^2 \sqrt{1+\log N(\delta \epsilon/2, B_{\mathcal{W}_{L, U}}(F, \delta), d_{\mathcal{W}})} \, \mathrm{d} \epsilon \\
        &\ge \frac{\epsilon_0}{2}\sqrt{1+\log\Big\lfloor\Big(\frac{16U}{27\delta^2}\Big)^{1/3}\Big\rfloor},
    \end{align*}
    where the inequality follows from the definition of $K$. This completes the proof of the proposition.
\end{proof}

The assumptions of Proposition~\ref{thm:second} also hold for a broad
class of distribution functions. Hence, Condition~\ref{con:A} fails for various values of \(m_\oplus\), even within \(\W_{L,U}\).

\section{Discussion}\label{sec:discussion}

The results of this paper identify a limitation of a commonly used
local entropy condition. Some studies have instead imposed weaker conditions in which the upper bound in Condition \ref{con:A} is allowed to depend on \(\delta\), rather than being required to be \(O(1)\). For example, \citet{ImJeonPark2025} considered an $O(\delta^{\alpha-1})$ bound for $0<\alpha\leq 1$. They showed that the convergence rate of their regression estimator depends on $\alpha$ using an empirical-process argument. They also proved that \(\W([a,b])\) satisfies this bound for any \(0 < \alpha < 1/2\). More recently, \citet{PetersenMuller2026} demonstrated that the global Fr\'echet regression can still achieve the optimal rate $O_p(n^{-1/2})$ via a projection argument rather than an empirical-process argument. 
Despite this result, obtaining a sharper upper bound for the integral in Condition \ref{con:A}, ideally one closer to the square-root logarithmic lower bound established here, remains an important open problem both for Fr\'echet regression beyond the global approach and for other statistical problems on the quadratic Wasserstein space that rely on empirical-process theory.

\bibliographystyle{apalike}
\bibliography{reference}

\end{document}